\documentclass[final,smallcondensed,numbook]{svjour3}

\usepackage[english]{babel}							
\usepackage[utf8]{inputenc}
\usepackage{amsmath,amssymb}
\usepackage{graphicx}
\usepackage{cancel,mathtools,empheq,latexsym}
\usepackage{subfig,epsfig,tikz,float}		        
\usepackage{booktabs,multicol,multirow,tabularx,array} 
\usepackage{bm}
\usepackage{cases}
\usepackage{mathrsfs}
\usepackage[normalem]{ulem}
\usepackage{hyperref}
\hypersetup{
    colorlinks=true,       
    linkcolor=red,          
    citecolor=blue,        
    filecolor=magenta,      
    urlcolor=cyan           
}

\newcommand{\ave}[1]{ \left\{\!\!\left\{ {#1} \right\}\!\!\right\} }
\newcommand{\jump}[1]{ \left[ \! \left[ {#1} \right] \! \right] }
\newcommand{\triple}[1]{ |\!|\!|{#1}|\!|\!| }
\newcommand{\bs}[1]{\boldsymbol{#1}}

\title{A boundary integral equation formulation for transient \\ electromagnetic transmission problems on Lipschitz domains}
\titlerunning{BIE formulation for transient electromagnetic transmission}
\dedication{}

\author{Tonatiuh S\'anchez-Vizuet
    \thanks{Partially funded by NSF grant NSF-DMS-2137305. }
    }
\institute{Tonatiuh S\'anchez-Vizuet
    \at Department of Mathematics, The University of Arizona, USA
        \email{tonatiuh@arizona.edu}
    }

\date{}

\begin{document}

\maketitle

\begin{abstract}

We propose a boundary integral formulation for the dynamic problem of electromagnetic scattering and transmission by homogeneous dielectric obstacles. In the spirit of Costabel and Stephan, we use the transmission conditions to reduce the number of unknown densities and to formulate a system of coupled boundary integral equations describing the scattered and transmitted waves. The system is transformed into the Laplace domain where it is proven to be stable and uniquely solvable. The Laplace domain stability estimates are then used to establish the stability and unique solvability of the original time domain problem. Finally, we show how the bounds obtained in both Laplace and time domains can be used to derive error estimates for semi discrete Galerkin discretizations in space and for fully discrete numerical schemes that use Convolution Quadrature for time discretization and a conforming Galerkin method for discretization of the space variables.   

\end{abstract}

\section{Introduction}

Consider a bounded open domain $\Omega_-\subset\mathbb R^3$ with Lipschitz boundary $\Gamma$ and define the complementary unbounded region as $\Omega_+:=\mathbb R^3\setminus\overline{\Omega_-}$. We will denote by $\bs n$ the unit vector normal to $\Gamma$ and pointing in the direction of $\Omega_+$. We will consider both $\Omega_-$ (the \textit{interior} domain or \textit{scatterer}) and $\Omega_+$ (the \textit{exterior} domain) to be filled with homogeneous materials with linear and constant electromagnetic properties encoded in their electric permittivity and magnetic permeability coefficients which will be denoted respectively by $\epsilon_+$ and $\mu_+$ in the exterior domain $\Omega_+$ and by  $\epsilon_-$ and $\mu_-$ in the interior domain $\Omega_-$. We will be concerned with the mathematical description of the situation where an electromagnetic wave propagating through $\Omega_+$ impinges upon $\Omega_-$. Due to the difference in electromagnetic properties between the two regions, part of the incident wave will be reflected back into $\Omega_-$ but part of it will get transmitted across the interface $\Gamma$. The incident electromagnetic perturbation consists an electric field ${\bf E}^{inc}$ and a magnetic field ${\bf H}^{inc}$, while---due to the assumption of linear behavior---the total field will be either the superposition of the incident fields with the scattered fields ${\bf E}^{scat}$ and ${\bf H}^{scat}$ in $\Omega_+$, or the transmitted fields  ${\bf E}$ and ${\bf H}$ inside $\Omega_-$. Hence, if the behavior of the scattered and transmitted fields is known, the total fields can be recovered by superposition as
\[
{\bf E}^{tot} = \begin{cases} {\bf E}^{scat} + {\bf E}^{inc} & \text{ in } \Omega_+ \\ {\bf E}  & \text{ in } \Omega_-  \end{cases}\,,\qquad \text{ and } \qquad 
{\bf H}^{tot} = \begin{cases} {\bf H}^{scat} + {\bf H}^{inc} & \text{ in } \Omega_+ \\ {\bf H}  & \text{ in } \Omega_- \end{cases}\, .
\]
Faraday's and Amp\`ere's laws enable the description of the electromagnetic field entirely in terms of either magnetic or electric field;we will opt for the latter option. Assuming that the incident wave has not yet impinged upon the scatterer at initial time, the interaction is then described by the system 
\begin{subequations}\label{eq:TDsystem}
\begin{alignat}{6}
\label{eq:TDsystema}
{\bf curl\, curl\, E}^{scat} + c_+^{-2}\partial_{tt}^2{\bf E} =\,& \bs 0 & \qquad \qquad & \text{ in } \Omega_+\times(0,T)\,,\\
\label{eq:TDsystemb}
{\bf curl\, curl\, E} + c_-^{-2}\partial_{tt}^2{\bf E} =\,& \bs 0 & \qquad \qquad & \text{ in } \Omega_-\times(0,T)\,,\\
\label{eq:TDsystemc}
\bs n\times \left({\bf E} - {\bf E}^{scat}\right)\times \bs n =\,& \bs n\times \left({\bf E}^{inc}\right)\times \bs n & \qquad \qquad & \text{ on } \Gamma\times(0,T)\,,\\
\label{eq:TDsystemd}
\bs n\times {\bf curl} \left(\mu_-^{-1}{\bf E} - \mu_+^{-1}{\bf E}^{scat}\right) =\,& \bs n\times {\bf curl}\left( \mu_+^{-1}{\bf E}^{inc} \right) & \qquad \qquad & \text{ on } \Gamma\times(0,T)\,,\\
\label{eq:TDsysteme}
 {\bf E}^{scat} = \partial_t{\bf E}^{scat} =\,& \bs 0 &\qquad \qquad & \text{ in } \Omega_+\times 0\,,\\
\label{eq:TDsystemf}
{\bf E} = \partial_t{\bf E} =\,& \bs 0 &\qquad \qquad & \text{ in } \Omega_-\times 0 .
\end{alignat}
\end{subequations}
where $c_+:=(\epsilon_+\mu_+)^{-1/2}$ and $c_-:=(\epsilon_-\mu_-)^{-1/2}$ denote the speed of propagation of the wave in $\Omega_+$ and $\Omega_-$ respectively. Equations \eqref{eq:TDsystema} and \eqref{eq:TDsystemb} are the electromagnetic wave equations in the interior and exterior domain, equations \eqref{eq:TDsystemc} and \eqref{eq:TDsystemd} express respectively the continuity of the tangential component of the electric and magnetic field and equations \eqref{eq:TDsysteme} and \eqref{eq:TDsystemf} express the fact that the incident wave is supported away of the interface $\Gamma$ at the initial time (i.e. the causality of the process). We note that the case when the coefficients in $\Omega_-$ are allowed to vary leads to a coupled PDE-BIE formulation analyzed by the authors in \cite{HsSaWe:2024}.

Due to the presence of an unbounded region in the model, reformulations of this problem that lead to boundary integral equations (BIEs) have been a tool of choice in computational electromagnetism for a long time. The main practical advantage of this approach is that no approximation of an unbounded domain is required in a simulation. This important property has led to the development of many different BIE formulations for the scattering problem (see \cite{HeKaRo2021} for a recent account of some of the options available). However, the pioneering formulation due to Stratton and Chu \cite{StCh1939} has withstood the test of time and remains one of the preferred options for practitioners to this date \cite{AnMaMa2019}. Stratton and Chu derived their celebrated representation under two important assumptions: that the interface between the scatterer and its surroundings is smooth, and that the time dependence of the electromagnetic radiation was time-harmonic.

Over time, the mathematical aspects surrounding BIEs for time-harmonic electromagnetic scattering by smooth obstacles developed into a mature theory (see, for instance \cite{CoKr2019,CoKr1983,Nedelec2001}) that has been successfully applied in many different contexts. However, the relaxation of the smoothness requirement was not achieved until the early 2000's when the work of Annalisa Buffa, Patrick Ciarlet Jr., Martin Costabel, Christoph Schwab and Dongwoo Sheen fully developed the theory for the time-harmonic Maxwell system in Sobolev spaces over Lipschitz domains \cite{BuCi2001A,BuCi2001B,BuCoSc2002,BuCoSh2002}. Given the success of the use of BIEs in the time-harmonic regime in many different areas of application, the development of the theory for the time domain system was somehow slower. However, the interest in the transient BIE formulation has received increased attention in the last decade, and some remarkable theoretical developments have been made by analyzing the system directly in the time domain \cite{QiSa2016}.

In this work, rather than analyzing the system directly in the time domain, we will take an alternative approach that involves the passage through the Laplace domain. The idea is to consider a Laplace-transformed version of the PDE system which is then recast in terms of BIEs---this will be done in Section \ref{sec:2}. To show the well posedness of this system we will have to consider an variational formulation of an equivalent transmission problem, which will be introduced in Section \ref{sec:3}. We will then show that said variational problem is well posed. in fact by considering test functions coming from a closed subspace of the energy space in we will simultaneously prove the well posedness of the continuous problem and of any conforming Galerkin discretization. By carefully tracking the role played by the Laplace variable in the stability constants of this system, it is then possible to extract information regarding the time-regularity of both problem data and the solutions of the problem along with explicit bounds for the growth of the solutions in time. The time domain results will be summarized in Section \ref{sec:4}. Thanks to poweful results by Sayas and Lubich, the bounds established in this fashion can then be used to state error bounds for semidiscrete numerical methods based on conforming Galerkin discretizations in space and also for full discretizations that make use of Convolution Quadrature in time. We will present the error estimates obtained in this way in the final Section \ref{sec:5}.

In the seminal article \cite{LaSa2009} Laliena and Sayas laid the groundwork for the use of these techniques in the analysis of transient wave phenomena. Since then, these ideas have been applied with success in acoustics \cite{BaRi2017,HaSa2016}, elastodynamics \cite{DoSaSa:2015,HsSaWe2022}, quantum mechanics \cite{MeRi2017}, and problems involving transient multi physics \cite{BrSaSa2018,HsSab:2020,HsSaSa2016,HsSaSaWe:2018,HsSa2021,HsSaWe2022,MeRiSa2022,SaSa:2016}. From the computational point of view, this approach also exploits the availability of frequency-domain computational codes that can easily be coupled with Convolution Quadrature \cite{Lubich1988,Lubich2004,LuSc1992} to produce time-domain results.  It is important to remark that different approaches in the Laplace domain have been employed. See for instance \cite{ChLy1996,Chudinovich97,Korikov2021,Lytova2003} and the references therein.

\section{A Laplace domain system of BIEs}\label{sec:2}
%
For the next several sections we will be exclusively working on the Laplace domain. In order to keep the notation as light as possible \textit{we will use the exact same notation for functions defined in the Laplace domain and their time-domain counterparts}. We will not deal with time domain variables until the final Section \ref{sec:5} and thus there shall be no ambiguity on whether a function is defined on either domain. \\

\noindent\textbf{Basic facts from Sobolev spaces for the Maxwell system.} Before posing the Laplace domain problem we will first introduce some essential notation and definitions from the theory of Sobolev spaces associated to the Maxwell system. We will keep the details to a minimum and refer the interested reader to \cite[Chapter 6]{HsWe2021} for a detailed account.

First of all, for an open domain $\mathcal O\subset\mathbb R^3$ with Lipschitz boundary $\Gamma:=\partial\mathcal O$, we define the space
 \[
 { \bf H} (\mathrm{{\bf curl}}, \mathcal{O}) :=  \{ \bs u \in {\bf L}^2 (\mathcal{O}) := \mathrm L^2( \mathcal{O})^3: {\bf curl}\,\bs u \in {\bf L}^2 (\mathcal{O}) \}
\]
which becomes a Hilbert space when endowed with the inner product
\[
(\bs u,\bs v)_{{\bf curl},\mathcal O} := \int_{\mathcal O} (\bs u \cdot \bs v + {\bf curl\,}\bs u \cdot {\bf curl\,}\bs v)\,, \qquad  \|{\bs u} \|^2 _{\mathrm{{\bf curl}},\; \mathcal{O}} := (\bs u,\overline {\bs u})_{{\bf curl},\mathcal O} = \| {\bs u}\|^2_{L^2(\mathcal{O})}+ \| \mathrm{{\bf curl}}\,{\bs  u} \|^2_{L^2(\mathcal{O})}.
\]
Two different notions of ``restriction to the boundary" $\Gamma$ are possible in this space. The \textit{tangential trace} $\gamma_t$
\[
 \gamma_t : { \bf H} (\mathrm{{\bf curl}}, \mathcal{O}) \rightarrow {\bf H}_{\|}^{-1/2}(\mathrm{div}_{\Gamma} , \Gamma) 
 \qquad \qquad \gamma_t: \bs u \mapsto \boldsymbol n \times \gamma\mathbf u\,,
\]
and the \textit{tangential projection} $\pi_t$ 
\[
\pi_t : 
  { \bf H} (\mathrm{ {\bf curl}}, \mathcal{O}) \rightarrow {\bf H}_{\perp}^{-1/2}(\mathrm{curl}_{\Gamma}, \Gamma) \qquad \qquad \pi_t : \bs u \mapsto \boldsymbol n \times (\gamma{\bs u} \times \boldsymbol n) = \gamma\bs u - (\gamma\bs u\cdot \boldsymbol n)\boldsymbol n.
\]
For the purpose of this work, the only things that we need to know about the trace spaces ${\bf H}_{\|}^{-1/2}(\mathrm{div}_{\Gamma} , \Gamma) $ and ${\bf H}_{\perp}^{-1/2}(\mathrm{curl}_{\Gamma}, \Gamma)$ are that they are both normed spaces equipped respectively with the graph norms
\[
\|\cdot\|_{\|}:= \|\cdot\|_{-1/2,{\rm div}_\Gamma,\Gamma} \qquad \text{ and } \qquad \|\cdot\|_{\perp}:= \|\cdot\|_{-1/2,{\rm curl}_\Gamma,\Gamma}\,,
\]
and that the former can be identified with the dual of the latter. We will denote the duality pairing of these spaces by 
\[
\langle\gamma\bs u,\pi_t\bs v\rangle_\Gamma = \int_\Gamma \gamma_t\bs u\cdot\pi_t\bs v\,.
\]
Both trace operators are continuous and surjective in their respective ranges, so that pseudo inverses $\gamma_t^\dagger$ and $\pi_t^\dagger$ from the trace spaces back into $ { \bf H} (\mathrm{{\bf curl}}, \mathcal{O})$ are well defined. We will commonly refer to the pseudo inverse mappings as \text{liftings}. In the particularly relevant case of $\mathcal O := \Omega_+ \cup \Omega_- = \mathbb R^3\setminus\Gamma$ we will use the superscripts $``+"$ and $``-"$ on the trace operators to explicitly indicate the domain from which the trace is being defined. In this case, the integration by parts formulas
\begin{equation}
\label{eq:ibpInt}
\langle\gamma_t^\pm{\bs u},\pi_t^\pm{\bs v}\rangle_\Gamma = \pm({\bs u},{\bf curl\, }\bs v)_{\Omega_\pm} \mp ({\bf curl\,}\bs u,{\bs v})_{\Omega_\pm}.
\end{equation} 
which hold for all ${\bs u},\,{\bs v}\in \mathbf H({\bf curl}, \Omega_-) $ and  ${\bs u},\,{\bs v}\in \mathbf H({\bf curl}, \Omega_+) $ will be of fundamental importance.\\

\noindent\textbf{Laplace domain PDE system.}  Considering the solutions of the system \eqref{eq:TDsystem} as elements of the Sobolev space
\[
\left\{\bs u \in {\bf H}({\bf curl}, \mathbb R^3\setminus\Gamma): {\bf curl\,curl\,}\bs u \in {\bf L}^2(\mathbb R^3\setminus\Gamma)\right\}\times\left\{\bs v \in {\bf H}({\bf curl}, \mathbb R^3\setminus\Gamma): {\bf curl\,curl\,}\bs v \in {\bf L}^2(\mathbb R^3\setminus\Gamma)\right\},
\]
and keeping in mind the causality of the scattered and transmitted waves \eqref{eq:TDsysteme} and \eqref{eq:TDsystemf}, it is possible to rigorously justify transforming the system into the Laplace domain \cite{Sayas2016}, which then becomes:
\begin{subequations}
\label{eq:LDsystem}
\begin{alignat}{6}
\label{eq:LDsystema}
{\bf curl}\,{\bf curl}\, {\bf E}^{scat} + (s/c_+)^2\, {\bf E}^{scat}  &=  {\bf 0}  & \quad & \text{ in } \Omega_+\,,\\
\label{eq:LDsystemb} 
{\bf curl}\,{\bf curl}\, {\bf E} + (s/c_-)^2\, {\bf E}  &= {\bf 0}  &\quad& \text{ in } \Omega_-\,,\\
\label{eq:LDsystemc}
\pi_t^-{\bf E} - \pi_t^+{\bf E}^{scat}  & = \pi_t^+{\bf E}^{inc} & \quad & \text{ on } \Gamma, \\
\label{eq:LDsystemd}
\mu^{-1}_-\gamma_t^-{ \bf curl}\,{\bf E} - \mu_+^{-1}\gamma_t^+{\bf curl\,E }^{scat} &   = \mu_+^{-1}\gamma_t^+{\bf curl } \, {\bf E}^{inc} & \quad & \text{ on } \Gamma,
\end{alignat}
\end{subequations}
 where the Laplace parameter $s$ is a complex number with positive real part. We will now reformulate the system above in terms of boundary integral equations and proceed to establish its well posedness and stability properties.\\
 
\noindent\textbf{Electromagnetic boundary integral operators and potentials.} As we recast the system \eqref{eq:LDsystem} in terms of boundary integral equations we will freely make use of notation and tools from potential theory. In the interest of brevity we will not devote too much time in the details and instead will only introduce the required notation; the interested reader is directed to \cite[Chapter 6]{HsWe2021} for the necessary background material in full detail. 

We will denote the fundamental solution of the Yukawa potential equation by
\[
G(x,y;s) := \frac{e^{-s|x-y|}}{4\pi|x-y|},
\]
and will use it to define the following electromagnetic layer potentials
\begin{alignat*}{6}
\nonumber
\mathcal D(s){\bf j} :=\,&{\bf curl}\!\! \int_\Gamma \!\!G\left(x,y;s\right){\bf j}(y)dy\,, &\qquad&& \mathcal S(s){\bf j}:=\,& {\bf curl\,curl}\!\!\int_\Gamma\!\! G\left(x,y;s\right){\bf j}(y)dy,\\
\widetilde{\mathcal D}(s)\mathbf m :=\,& \mathcal D(s)\circ\gamma_t\circ\pi_t^\dagger\,\mathbf m\,, &\qquad \qquad \qquad && \widetilde{\mathcal S}(s)\mathbf m :=\,& \mathcal S(s)\circ\gamma_t\circ\pi_t^\dagger\mathbf m\,,
\end{alignat*} 
where $\mathbf m\in {\bf H}^{-1/2}_{\perp} ( {\rm curl}_{\Gamma}, \Gamma)$ and ${\bf j} \in {\bf H}^{-1/2}_{||} ({\rm div}_{\Gamma}, \Gamma )$, are referred to as \textit{density functions}. We will also make use of the following boundary integral operators
\begin{alignat*}{6}
\nonumber
\mathcal K(s) \mathbf j :=\,& \gamma_t{\bf curl} \int_\Gamma G(x,y;s){\bf j}(y)dy, & \qquad \qquad && \mathcal V(s) {\bf j} :=\,& \pi_t{\bf curl\,curl}\int_\Gamma G(x,y;s){\bf j}(y)dy, \\
 \widetilde{\mathcal K}(s) \mathbf m :=&\, \mathcal K(s)\circ\gamma_t\circ\pi_t^\dagger{\bf m},  & \qquad && \widetilde{\mathcal V}(s){\bf m}:=\,&\mathcal V(s)\circ\gamma_t\circ\pi_t^\dagger{\bf m}.
\end{alignat*}
The traces of the layer potentials above satisfy the identities
\begin{equation}
\label{eq:LPtraces}
\gamma_t^{\pm}\mathcal D(s) = \pm\frac{1}{2} + \mathcal K(s)
\quad \text{ and } \quad  
 \pi_t^{\pm}\widetilde{\mathcal D}(s) = \pm\frac{1}{2} + \widetilde{\mathcal K}(s),
\end{equation}
and have the following continuity/discontinuity properties across the interface $\Gamma$
\begin{equation}
\label{eq:JumpAndAve}
\begin{array}{llll}
\quad \ave{\gamma_t\mathcal D(s){\bf j}} = \mathcal K(s){\bf j}\,, & \quad \jump{\gamma_t\mathcal D(s){\bf j}\,} = -{\bf j}\,, & \quad \ave{\pi_t\widetilde{\mathcal D}(s){\bf m}} = \widetilde{\mathcal K}(s){\bf m}\,, & \quad \jump{\pi_t\widetilde{\mathcal D}(s){\bf m}} = -{\bf m}\,,\\[2ex]
\quad \ave{\pi_t \mathcal S(s){\bf j}} = \mathcal V(s){\bf j}\,, & \quad \jump{\pi_t \mathcal S(s){\bf j}\,}  = {\bf 0}\,,& \quad \ave{\gamma_t\widetilde{\mathcal S}(s){\bf m}} = \widetilde{\mathcal V}(s){\bf m}\,, & \quad \jump{\gamma_t\widetilde{\mathcal S}(s){\bf m}} = {\bf 0}\,.
\end{array}
\end{equation} 
Above, as is customary, we denoted the jump and average of the traces of a function $\bs u$ across $\Gamma$ by
\[
\jump{\gamma_t\bs u} : = \gamma_t^-\bs u - \gamma_t^+ \bs u \qquad \text{ and } \qquad \ave{\gamma_t\bs u} : = \frac{1}{2}\left( \gamma_t^-\bs u + \gamma_t^+ \bs u\right),
\]
and similarly for the tangential projection $\pi_t$.\\

\noindent\textbf{Laplace domain BIE.} We now propose solutions of the system \eqref{eq:LDsystem}, ${\bf E}^{scat}$ and ${\bf E}$, in terms of boundary layer potentials of the form
\[
{\bf E}^{scat} := \widetilde{\mathcal D}(s/c_+)\mathbf m^{scat} - \frac{1}{s\epsilon_+}\mathcal S(s/c_+)\mathbf j^{scat} \qquad \text{ and } \qquad {\bf E} :=  \frac{1}{s\epsilon_-}\mathcal S(s/c_-)\mathbf j - \widetilde{\mathcal D}(s/c_-)\mathbf m\,,
\]
for unknown densities ${\bf m}^{scat}\,,{\bf m}\in  {\bf H}^{-1/2}_{\perp} ( {\rm curl}_{\Gamma}, \Gamma)$ and ${\bf j}^{scat},\,{\bf j} \in {\bf H}^{-1/2}_{||} ({\rm div}_{\Gamma}, \Gamma )$. From their definition in terms of layer potentials, it is possible to show (see for instance \cite[Section 4.3]{HsSaWe:2024}) that these functions indeed satisfy equations \eqref{eq:LDsystema} and \eqref{eq:LDsystemb} respectively, and that their curls can be expressed as
\[
{\bf curl\, E}^{scat} = \widetilde{\mathcal S}(s/c_+)\mathbf m^{scat} + s\mu_+\mathcal D(s/c_+)\mathbf j^{scat} \qquad \text{ and } \qquad {\bf curl\, E} :=  -\widetilde{\mathcal S}(s/c_-)\mathbf m - s\mu_-\mathcal D(s/c_-)\mathbf j\,.
\]
Given these integral representations and the observation that, since $supp\,{\bf E}^{scat}\subset \Omega_+$ and $supp\,{\bf E}\subset \Omega_-$, $\pi_t^-{\bf E}^{scat} = \gamma_t^-{\bf curl\,E}^{scat}=\pi_t^+{\bf E} = \gamma_t^+{\bf curl\,E}=\bs 0$, the properties \eqref{eq:JumpAndAve} imply that
\[
\begin{array}{lll}
{\bf m}^{scat} = -\jump{\pi_t{\bf E}^{scat}} = \pi_t^+{\bf E}^{scat}\,, & \phantom{+++}& s\mu_+{\bf j}^{scat} = -\jump{\gamma_t{\bf curl\, E}^{scat}} = \gamma_t^+{\bf curl\, E}^{scat}\,,\\[.8ex]
{\bf m} = \jump{\pi_t{\bf E}^{\phantom{t}}} = \pi_t^-{\bf E}\,, &\phantom{+++} & s\mu_-{\bf j} = \jump{\gamma_t{\bf curl\, E}^{\phantom{t}}} = \gamma_t^-{\bf curl\, E}\,.
\end{array}
\]
Combining the identities above with the transmission conditions \eqref{eq:LDsystemc} and \eqref{eq:LDsystemd} we can eliminate one pair of unknown densities by observing that
\[
{\bf m}^{scat} = {\bf m} - \bs\phi \qquad \text{ and } \qquad {\bf j}^{scat} = {\bf j} -(s\mu_+)^{-1}\bs\lambda,
\]
where, to simplify the notation, we have defined
\[
\bs \phi :=\pi_t^+{\bf E}^{inc} \qquad \text{ and } \qquad \bs \lambda:= \gamma_t^+ {\bf curl\, E}^{inc}.
\]
This leads to the representations
\begin{subequations}
\label{eq:IntegralRepresentations}
\begin{alignat}{8}
\label{eq:IntegralRepresentationsA}
{\bf E}^{scat} =\,& \widetilde{\mathcal D}(s/c_+)\left({\bf m} - \bs\phi\right) -(s\epsilon_+)^{-1}\mathcal S(s/c_+)\left({\bf j} -(s\mu_+)^{-1}\bs\lambda\right)\,, \\
\label{eq:IntegralRepresentationsB}
{\bf curl\, E}^{scat} =\,& \widetilde{\mathcal S}(s/c_+)\left({\bf m} - \bs\phi\right) + s\mu_+\mathcal D(s/c_+)\left({\bf j} -(s\mu_+)^{-1}\bs\lambda\right)\,,\\
\label{eq:IntegralRepresentationsC}
{\bf E} =\,&  (s\epsilon_-)^{-1}\mathcal S(s/c_-)\mathbf j - \widetilde{\mathcal D}(s/c_-)\mathbf m\,,\\
\label{eq:IntegralRepresentationsD}
{\bf curl\, E} =\,&  -\mathcal S(s/c_-)\mathbf m - s\mu_-\mathcal D(s/c_-)\mathbf j\,.
\end{alignat}
\end{subequations}
From the integral representations above, we can use the properties \eqref{eq:LPtraces} and \eqref{eq:JumpAndAve} to recast the transmission conditions \eqref{eq:LDsystemc} and \eqref{eq:LDsystemd} in terms of boundary integral operators, arriving at the following system of BIEs in the style of that of Costabel and Stephan \cite{CoSt1985}
\[
\mathbb L(s)({\bf j},{\bf m})^\top = \mathbb R(s) (\bs \lambda,\bs \phi)^\top\,,
\]
 where we have defined the matrices of operators
\begin{subequations}
\label{eq:Matrices}
\begin{align}
\label{eq:MatricesA}
\mathbb L(s) :=\,&
\begin{bmatrix}
(s\epsilon_+)^{-1}\mathcal V(s/c_+) - (s\epsilon_-)^{-1}\mathcal V(s/c_-) &\phantom{++} & \widetilde{\mathcal K}(s/c_+) + \widetilde{\mathcal K}(s/c_-)  \\[1ex]
s\left(\mathcal K(s/c_+) + \mathcal K(s/c_-) \right) & \phantom{++} & \mu_+^{-1}\widetilde{\mathcal V}(s/c_+) + \mu_-^{-1}\widetilde{\mathcal V}(s/c_-)
\end{bmatrix}\,,\\[2ex]
\label{eq:MatricesB}
\mathbb R(s) :=\,&
\begin{bmatrix}
(c_+/s)^2\mathcal V(s/c_+) & \phantom{++} & -\frac{1}{2} + \widetilde{\mathcal K}(s/c_+)\\[1ex]
\mu_+^{-1}\left(-\frac{1}{2}+ \mathcal K(s/c_+)\right) & \phantom{++} & \mu_+^{-1}\widetilde{\mathcal V}(s/c_+)
\end{bmatrix}\,.
\end{align}
\end{subequations}
We will now study the well-posedness of this system via a conforming Galerkin discretization. To that end, we will consider the following \textit{closed subspaces} of the solution spaces
\[
{\bf X}_h \subset {\bf H}_{\|}^{-1/2}({\rm div}_\Gamma,\Gamma)  \qquad \text{ and } \qquad {\bf Y}_h \subset {\bf H}_\perp^{-1/2}({\rm curl}_\Gamma,\Gamma)\,,
\]
as well as their polar sets, defined respectively as
\begin{align*}
{\bf X}_h^\circ :=\,& \left\{\bs\varphi\in  {\bf H}_\perp^{-1/2}({\rm curl}_\Gamma,\Gamma): \langle\bs \eta ,\bs\varphi\rangle_\Gamma = 0 \;\;\forall\bs\eta\in {\bf X}_h \right\} \,,\\
{\bf Y}_h^\circ :=\,& \left\{\bs\eta\in  {\bf H}_{\|}^{-1/2}({\rm div}_\Gamma,\Gamma) : \langle\bs \eta ,\bs\varphi\rangle_\Gamma = 0 \;\;\forall\bs\varphi\in {\bf Y}_h \right\}\,.
\end{align*}
The weak formulation that we will focus on can then be stated as follows:
\begin{equation}
\label{eq:WeakBIE}
\left.\begin{array}{l}
\text{ Given  data }\,(\bs \lambda,\bs\phi)\in {\bf H}_{\|}^{-1/2}({\rm div}_\Gamma,\Gamma) \times {\bf H}_\perp^{-1/2}({\rm curl}_\Gamma,\Gamma)\,,\, \\[1.2ex]
\text{ find }\, ({\bf j}_h, {\bf m}_h) \in {\bf X}_h\times{\bf Y}_h \; \text{ such that: } \phantom{++}\\[1.2ex]
\langle \mathbb L(s)({\bf j}_h,{\bf m}_h)^\top, (\bs\eta,\bs\varphi)\rangle_\Gamma = \langle\mathbb R(s) (\bs \lambda,\bs \phi)^\top, (\bs\eta,\bs\varphi)\rangle_\Gamma \;\; \forall (\bs\eta,\bs\varphi)\in  {\bf X}_h\times{\bf Y}_h\,, \phantom{++}
\end{array} \right\}
\end{equation}
where the matrices of operators $\mathbb L(s)$ and $\mathbb R(s)$ have been defined in \eqref{eq:Matrices}. In the following sections we will prove the well-posedness of the problem above and will transfer this result into time domain estimates for the solutions of the time-domain problem \eqref{eq:TDsystem}.
%
%
\section{Well posedness in the Laplace domain}\label{sec:3}

\noindent\textbf{An equivalent transmission problem.} We could prove the unique solvability of the weak formulation \eqref{eq:WeakBIE} by studying the coercivity of the Laplace-domain Calder\'on operator in a way similar to the analysis performed in \cite{Nick2022} for the scattering off screens. Instead, we will prefer the technique of Laliena and Sayas \cite{LaSa2009} that requires the study of a non-standard transmission problem that underscores the nature of functions defined in terms of layer potentials as fundamentally defined over the entire space. This idea can be traced back to N\'edelec and Planchard \cite{NePl1973}, who used it in the
 context of first-kind integral equations. When compared to \cite{Nick2022}, the current manuscript presents an alternate way of analyzing the boundary integral equations for electromagnetic scattering in the time domain. In particular, the use of an equivalent non-standard transmission problem to analyze the well-posedness of the problem, leverages tools common in the field of linear elliptic PDE's. This makes the arguments accesible to the broad community of numerical analysts who are typically well versed in these types of arguments.

Indeed, if the \textit{exact} solutions to \eqref{eq:WeakBIE} are known and used in the representations \eqref{eq:IntegralRepresentations} to define scattered and transmitted fields, the resulting functions will be supported in either $\Omega_-$ or $\Omega_+$. However, if the Galerkin approximations $({\bf j}_h,{\bf m}_h)$ are used instead, the resulting fields need not be supported entirely on one side of $\Gamma$. This observation leads to the following proposition.

\begin{proposition}
\label{pr:1} 
Let $({\bf j}_h,{\bf m}_h)$ be a solution pair to \eqref{eq:WeakBIE}. Then ${\bf E}_h^{scat}$ and ${\bf E}_h$ are defined using these densities through the integral representations \eqref{eq:IntegralRepresentations} if and only if they satisfy the following transmission problem
\begin{subequations}
\label{eq:EquivalentTP}
\begin{align}
\label{eq:EquivalentTPA}
{\bf curl\,curl\, E}_h^{scat} + (s/c_+)^2{\bf E}_h^{scat} =\,& \bs 0  \qquad \qquad \qquad \text{ in } \mathbb R^3\setminus\Gamma\,, \\
\label{eq:EquivalentTPB}
{\bf curl\,curl\, E}_h + (s/c_-)^2{\bf E}_h =\,& \bs 0  \qquad \qquad  \qquad \text{ in } \mathbb R^3\setminus\Gamma\,, \\
\label{eq:EquivalentTPC}
\jump{\pi_t{\bf E}_h^{scat}} + \jump{\pi_t{\bf E}_h^{\phantom{t}}} =\,& \bs\phi \,,\\
\label{eq:EquivalentTPD}
\mu_+^{-1}\jump{\gamma_t{\bf curl\,E}_h^{scat}} + \mu_-^{-1}\jump{\gamma_t{\bf curl\,E}_h^{\phantom{t}}} =\,& \mu_+^{-1}\bs\lambda \,,\\
\label{eq:EquivalentTPE}
\left(\jump{\gamma_t{\bf curl\,E}_h^{\phantom{t}}}, \jump{\pi_t{\bf E}_h^{scat}}\right) \in\,& {\bf X}_h\times{\bf Y}_h\,, \\
\label{eq:EquivalentTPF}
\left(\pi_t^+{\bf E}_h - \pi_t^-{\bf E}_h^{scat}\,,\,\, \mu_-^{-1}\gamma_t^+{\bf curl\,E}_h - \mu_+^{-1}\gamma_t^-{\bf curl\,E}_h^{scat} \right) \in\,& {\bf X}_h^\circ \times {\bf Y}_h^\circ\,.
\end{align}
\end{subequations} 

\end{proposition}
\begin{proof}
If ${\bf j}_h$ and ${\bf m}_h$ satisfy \eqref{eq:WeakBIE} and are used to define ${\bf E}_h^{scat}$ and ${\bf E}_h$ through the integral representations \eqref{eq:IntegralRepresentations}, then clearly \eqref{eq:EquivalentTPA}
 and \eqref{eq:EquivalentTPB} are satisfied. Moreover, the jumps of the traces of ${\bf E}_h$ are given by
\[
\jump{\gamma_t{\bf curl\, E}_h^{\phantom{t}}} = s\mu_-{\bf j}_h \in {\bf X}_h \qquad \text{ and } \qquad \jump{\pi_t{\bf E}_h^{\phantom{t}}} = {\bf m}_h \in {\bf Y}_h \,,
\]
and therefore \eqref{eq:EquivalentTPE} holds. From these equalities it also follows that for ${\bf E}_h^{scat}$ we have
\[
\jump{\pi_t{\bf E}_h^{scat}} = \bs\phi - {\bf m}_h = \bs\phi - \jump{\pi_t{\bf E}_h^{\phantom{t}}} \qquad \text{ and } \qquad \jump{\gamma_t{\bf curl\, E}_h^{scat}} = \bs\lambda - s\mu_+{\bf j}_h = \bs\lambda - (\mu_+/\mu_-)\jump{\gamma_t{\bf curl\, E}_h^{\phantom{t}}}\,
\]
which imply \eqref{eq:EquivalentTPC} and \eqref{eq:EquivalentTPD}. Finally, using the equalities above to write ${\bf j}_h$ and ${\bf m}_h$ in terms of the jumps, it is easy to see that the weak equation in \eqref{eq:WeakBIE} is equivalent to \eqref{eq:EquivalentTPF}.

Conversely, if $({\bf E}_h^{scat}, {\bf E}_h)$ satisfy the transmission problem \eqref{eq:EquivalentTP}, we define
\[
{\bf j}_h := -(s\mu_-)^{-1}\jump{\gamma_t{\bf curl\, E}_h^{\phantom{t}}} \in {\bf X}_h \qquad \text{ and } \qquad {\bf m}_h := \jump{\pi_t{\bf E}_h^{\phantom{t}}} \in {\bf Y}_h  \,.
\]
We then use \eqref{eq:EquivalentTPC} and \eqref{eq:EquivalentTPD} to express $\jump{\pi_t{\bf E}_h^{scat}}$ and $\jump{\pi_t{\bf E}_h^{scat}}$ in terms of ${\bf m}_h$, ${\bf j}_h$, and the problem data $\bs \lambda$ and $\bs \phi$. With these definitions, equations \eqref{eq:EquivalentTPA} and \eqref{eq:EquivalentTPB} imply that that the integral representations \eqref{eq:IntegralRepresentations} hold. Hence, starting from the integral representations and writing \eqref{eq:EquivalentTPF} in terms of boundary integral operators with the aid of \eqref{eq:LPtraces} and \eqref{eq:JumpAndAve}, we arrive at the weak BIE in \eqref{eq:WeakBIE}. $\Box$
\end{proof}  
 
\noindent\textbf{Variational formulation.} As a consequence of the preceding proposition, we will infer the unique solvability of the system of BIEs \eqref{eq:WeakBIE} from that of the transmission problem. We will do so by studying the variational form of \eqref{eq:EquivalentTP} and will start with some definitions.

We start by introducing an energy norm for functions in ${\bf H}({\bf curl},\mathbb R^3\setminus\Gamma)$
\[
\triple{\bs u}_{s}^2 := \| {\bf curl\,} \bs u\|^2 + |s|^2\|\bs u\|^2 = \left({\bf curl\,}\bs u,{\bf curl\,}\overline{\bs u}\right)_{\mathbb R^3} + |s|^2 \left(\bs u,\overline{\bs u}\right)_{\mathbb R^3}\,,
\]
where the bar on top of the second argument of the inner products denotes complex conjugation, and the norm in ${\bf L}^2(\mathbb R^3\setminus\Gamma)$ was denoted simply by $\|\cdot\|$. Clearly, $\triple{\bs u}_{1}$ coincides with the usual norm in the space ${\bf H}({\bf curl},\mathbb R^3\setminus\Gamma)$. Moreover, for a complex number $s$ with positive real part $\sigma := {\rm Re}\, s > 0$, the following equivalence relations hold
\begin{equation}
\label{eq:NormEquiv}
\underline{\sigma} \triple{\bs u}_{1} \leq \triple{\bs u}_{s} \leq \frac{|s|}{\underline{\sigma}}\triple{\bs u}_1,
\end{equation}
where we have denoted $\underline{\sigma}:=\min\{1,\sigma\}$ and used the easily verifiable fact that if $\sigma = {\rm Re}\,s >0$ then $\max\{1,|s|\}\underline{\sigma}\leq |s|$. We now define the bilinear form ${\rm B}\left(\cdot,\cdot\right): \left({\bf H}({\bf curl},\mathbb R^3\setminus\Gamma)\right)^2 \rightarrow \mathbb C$ as
\[
{\rm B}\left((\bs u,\bs v), (\widetilde{\bs u},\widetilde{\bs v})\right): = \mu_+^{-1}\Big(\left({\bf curl\,}\bs u,{\bf curl\,}\widetilde{\bs u}\right)_{\mathbb R^3} +  (s/c_+)^2\left(\bs u,\widetilde{\bs u}\right)_{\mathbb R^3}\Big) + \mu_-^{-1}\Big(\left({\bf curl\,}\bs v,{\bf curl\,}\widetilde{\bs v}\right)_{\mathbb R^3} +  (s/c_-)^2\left(\bs v,\widetilde{\bs v}\right)_{\mathbb R^3}\Big)\,
\]
and remark that simple algebraic computations lead to
\begin{subequations}
\label{eq:NormProperties}
\begin{align}
\label{eq:NormPropertiesA}
\left|{\rm B}\left((\bs u,\bs v), (\widetilde{\bs u},\widetilde{\bs v})\right)\right| \lesssim\,& \triple{(\bs u, \bs v)}_s \triple{(\widetilde{\bs u}, \widetilde{\bs v})}_s\, \\
\label{eq:NormPropertiesB}
\sigma\triple{(\bs u, \bs v)}_s^2\lesssim \,& {\rm Re}\left(\,\overline s\,{\rm B}\left((\bs u,\bs v), (\overline{\bs u},\overline{\bs v})\right) \right)\,,  
\end{align}
\end{subequations}
where the symbol $\lesssim$ hides constants that may depend only on the geometry of $\Gamma$ and the physical parameters but not on $s$, and we have denoted $\triple{(\bs u, \bs v)}_s^2 := \triple{\bs u}_s^2 + \triple{\bs v}_s^2$.

The discrete space for the variational formulation will be
\[
\mathbb H^h :=\left\{(\bs u,\bs v)\in\left({\bf H}({\bf curl}^2,\mathbb R^3\setminus\Gamma)\right)^2 :  \jump{\gamma_t\bs v}\in {\bf X}_h \;\; \text{ and } \;\;\jump{\pi_t\bs u}\in {\bf Y}_h \right\}\,,
\]
where
\[
{\bf H}({\bf curl}^2,\mathbb R^3\setminus\Gamma) :=\left\{\bs u \in {\bf H}({\bf curl},\mathbb R^3\setminus\Gamma) : {\bf curl\,curl\,}\bs u\in {\bf L}^2(\mathbb R^3\setminus\Gamma) \right\}\,.
\]
Finally, we need to define a generalized trace operator that will account for the essential boundary conditions stemming from \eqref{eq:EquivalentTPC} and the first component of \eqref{eq:EquivalentTPF}
\begin{alignat*}{6}
\pi :\;&& \left({\bf H}({\bf curl},\mathbb R^3\setminus\Gamma)\right)^2 & \;\; \longrightarrow \;\; && {\bf X}_h^\prime\times{\bf H}_\perp^{-1/2}({\bf curl}_\Gamma,\Gamma) \\
\nonumber
 && (\bs u,\bs v) & \;\; \longmapsto \;\; && \left({\bf P}_{{\bf X}_h}^\top\left(\pi_t^+\bs v - \pi_t^-\bs u\right), \jump{\pi_t\bs u} + \jump{\pi_t\bs v}\right)\,,
\end{alignat*}
Where the operator ${\bf P}_{{\bf X}_h}^\top:  {\bf H}_\perp^{-1/2}({\rm curl}_\Gamma,\Gamma) \to {\bf X}_h^\prime$ restricts the action of its argument to elements of ${\bf X}_h$. Note that, owing to its definition in terms of the trace projection operator $\pi_t$, the generalized trace $\pi$ has a well defined pseudo inverse $\pi^\dagger$. Moreover, the discrete kernel of $\pi$ will be of particular relevance and we shall denote it by
\[
\mathbb H_0^h : = \left\{(\bs u,\bs v)\in  \left({\bf H}({\bf curl},\mathbb R^3\setminus\Gamma)\right)^2 : \pi_t^+\bs v - \pi_t^-\bs u \in {\bf X}_h^\circ \;,\; \jump{\pi_t\bs u} = - \jump{\pi_t\bs v} \in {\bf Y}_h \right\}\,.
\]
With all these definitions in place we can now state the variational form of the transmission problem \eqref{eq:EquivalentTP} as follows:
\begin{subequations}
\label{eq:VariationalForm}
\begin{align}
\nonumber
& \text{Given data }\,(\bs \lambda,\bs\phi)\in {\bf H}_{\|}^{-1/2}({\rm div}_\Gamma,\Gamma) \times {\bf H}_\perp^{-1/2}({\rm curl}_\Gamma,\Gamma)\,,\,  \phantom{+}\\[-.4ex]
\nonumber
& \text{find } ({\bf E}_h^{scat},{\bf E}_h)\in \mathbb H^h \text{ such that: } \\[-.4ex]
\label{eq:VariationalFormA}
& \pi ({\bf E}_h^{scat},{\bf E}_h) = \left(\bs 0,\bs\phi\right) \qquad \text{ and }\\[-.4ex]
\label{eq:VariationalFormB}
& {\rm B}\left(({\bf E}_h^{scat},{\bf E}_h),(\bs u,\bs v)\right) = -\langle\mu_+^{-1}\bs\lambda,\pi_t^+\bs u\rangle_\Gamma \quad \forall\,(\bs u,\bs v)\in \mathbb H_0^h\,.
\end{align}
\end{subequations}

\begin{proposition}
\label{pr:2}
A pair of functions $({\bf E}_h^{scat},{\bf E}_h)\in \left({\bf H}({\bf curl}^2,\mathbb R^3\setminus\Gamma)\right)^2$ are a solution of the transmission problem
\eqref{eq:EquivalentTP} if and only if they satisfy \eqref{eq:VariationalForm}.
\end{proposition}
\begin{proof}
Let $({\bf E}_h^{scat},{\bf E}_h)$ be a solution pair to \eqref{eq:EquivalentTP} and $\bs u,\bs v \in \mathbb H_0^h$. We first observe that \eqref{eq:EquivalentTPE} guarantees that $({\bf E}_h^{scat},{\bf E}_h)\in\mathbb H^h$. Moreover, given that $({\bf E}_h^{scat},{\bf E}_h)$ satisfy \eqref{eq:EquivalentTPA} and \eqref{eq:EquivalentTPB}, splitting the integrals over $\mathbb R^3$ as the sum of integrals over $\Omega_-$ and $\Omega_+$ and using the integration by parts formulas \eqref{eq:ibpInt} the bilinear form can be expressed as: 
{\small \begin{align*}
{\rm B}\!\left(\!({\bf E}_h^{scat}\!,{\bf E}_h), (\bs u,\bs v)\!\right) :=\,& \mu_+^{-1}\Big(\!\!\left({\bf curl\,E}_h^{scat},{\bf curl\,}\bs u\right)_{\mathbb R^3}\! + \! (s/c_+)^2\!\left({\bf E}_h^{scat},\bs u\right)_{\mathbb R^3}\!\!\Big) \!+\! \mu_-^{-1}\Big(\!\!\left({\bf curl\,E}_h,{\bf curl\,}\bs v\right)_{\mathbb R^3} \!+ \! (s/c_-)^2\!\left({\bf E}_h,\bs v\right)_{\mathbb R^3}\!\!\Big) \\
=\,& \mu_+^{-1}\!\Big(\!\!\left({\bf curl\,E}_h^{scat}\!,{\bf curl\,}\bs u\right)_{\mathbb R^3}\! - \!\left({\bf curl\,curl\,E}_h^{scat}\!,\bs u\right)_{\mathbb R^3}\!\!\Big) \!+\! \mu_-^{-1}\!\Big(\!\!\left({\bf curl\,E}_h,{\bf curl\,}\bs v\right)_{\mathbb R^3} \! - \! \left({\bf curl\,curl\,E}_h,\bs v\right)_{\mathbb R^3}\!\!\Big) \\
=\,& -\left\langle\mu_+^{-1}\jump{\gamma_t{\bf E}_h^{scat}} + \mu_-^{-1}\jump{\gamma_t{\bf E}_h^{\phantom{t}}},\pi_t^+\bs u\right\rangle_\Gamma - \;\left\langle\mu_+^{-1}\gamma_t^-{\bf E}_h^{scat} - \mu_-^{-1}\gamma_t^+{\bf E}_h^{\phantom{t}}, \jump{\pi_t\bs u}\right\rangle_\Gamma \\
& - \left\langle -\mu_-^{-1}\gamma_t^+{\bf curl\, E}, \jump{\pi_t\bs u} + \jump{\pi_t\bs v}\right\rangle_\Gamma +\; \left\langle -\mu_-^{-1}\jump{\gamma_t{\bf curl\, E}}, \pi_t^+\bs v - \pi_t^-\bs u\right\rangle_\Gamma\,.
\end{align*} }
Focusing on the boundary terms at the end of the expression above we notice that:
\begin{enumerate}
\item The first term equals $-\langle\lambda,\pi_t^+\bs u\rangle_\Gamma$ due to \eqref{eq:EquivalentTPD}.
\item The second term vanishes as a consequence of $\jump{\pi_t\bs u}\in {\bf Y}_h$ and the second component of \eqref{eq:EquivalentTPF}.
\item The third term vanishes since $ \jump{\pi_t\bs u} + \jump{\pi_t\bs v} = \bs 0$.
\item The fourth term vanishes since $\pi_t^+\bs v - \pi_t^-\bs u \in {\bf X}_h^\circ$ and $\jump{\gamma_t{\bf curl\, E}}\in {\bf X}_h$ by the first component of \eqref{eq:EquivalentTPE}.
\end{enumerate}
These points prove that indeed \eqref{eq:VariationalFormB} is satisfied. Finally, we see that  \eqref{eq:EquivalentTPC} and the first component of \eqref{eq:EquivalentTPF} guarantee that $\pi ({\bf E}_h^{scat},{\bf E}_h) = \left(\bs 0,\bs\phi\right)$ and thus $({\bf E}_h^{scat},{\bf E}_h)$ satisfy \eqref{eq:VariationalForm}.

On the other hand, if $({\bf E}_h^{scat},{\bf E}_h)\in \mathbb H^h$ is a solution pair to \eqref{eq:VariationalForm}, then $\pi ({\bf E}_h^{scat},{\bf E}_h) = \left(\bs 0,\bs\phi\right)$ implies both \eqref{eq:EquivalentTPC} and the first component of \eqref{eq:EquivalentTPF}, while \eqref{eq:EquivalentTPE} is ingrained in $\mathbb H^h$. Using the integration by parts formulas \eqref{eq:ibpInt}, and the fact that
\[
\left\langle\gamma_t^+\bs u,\pi_t^+\widetilde{\bs u} \right\rangle_\Gamma - \left\langle\gamma_t^-\bs u,\pi_t^-\widetilde{\bs u} \right\rangle_\Gamma  = -\left\langle\jump{\gamma_t\bs u},\pi_t^-\widetilde{\bs u} \right\rangle_\Gamma - \left\langle\gamma_t^+\bs u,\jump{\pi_t\widetilde{\bs u}} \right\rangle_\Gamma
\]
we can obtain
\begin{align}
\nonumber
{\rm B}\left((\bs u,\widetilde{\bs u}), (\bs v,\widetilde{\bs v})\right) :=\,& \mu_+^{-1}\Big(\left({\bf curl\, curl\,}\bs u,\widetilde{\bs u}\right)_{\mathbb R^3} +  (s/c_+)^2\left(\bs u,\widetilde{\bs u}\right)_{\mathbb R^3}\Big) + \mu_-^{-1}\Big(\left({\bf curl\, curl\,}\bs v,\widetilde{\bs v}\right)_{\mathbb R^3} +  (s/c_-)^2\left(\bs v,\widetilde{\bs v}\right)_{\mathbb R^3}\Big) \\
\nonumber
& -\mu_+^{-1}\Big( \left\langle\jump{\gamma_t\bs u},\pi_t^-\widetilde{\bs u} \right\rangle_\Gamma + \left\langle\gamma_t^+\bs u,\jump{\pi_t\widetilde{\bs u}} \right\rangle_\Gamma   \Big) - \mu_-^{-1}\Big(\left\langle\jump{\gamma_t\bs v},\pi_t^-\widetilde{\bs v} \right\rangle_\Gamma + \left\langle\gamma_t^+\bs v,\jump{\pi_t\widetilde{\bs v}} \right\rangle_\Gamma\Big)\,.
\end{align}
Thus, equation \eqref{eq:VariationalFormB} implies that
\begin{align}
\nonumber
\mu_+^{-1}\left( \left({\bf curl\,curl\,E}_h^{scat},\bs u\right)_{\mathbb R^3}\right.& +   \left.(s/c_+)^2\left({\bf E}_h^{scat},\bs u\right)_{\mathbb R^3}\right)
+ \mu_-^{-1}\left( \left({\bf curl\,curl\, E}_h,\bs v\right)_{\mathbb R^3} +  (s/c_-)^2\left({\bf E}_h,\bs v\right)_{\mathbb R^3}\right) \\
\nonumber
 =\,&\;\; \mu_+^{-1}\Big( \left\langle\jump{\gamma_t{\bf curl\,E}_h^{scat}},\pi_t^-\bs u \right\rangle_\Gamma + \left\langle\gamma_t^+{\bf curl\,E}_h^{scat},\jump{\pi_t\bs u} \right\rangle_\Gamma   \Big)  -\langle\mu_+^{-1}\bs\lambda,\pi_t^+\bs u\rangle_\Gamma \\
 \label{eq:ExpandedB}
& + \mu_-^{-1}\Big(\left\langle\jump{\gamma_t{\bf curl\,E}_h^{\phantom{t}}},\pi_t^-\bs v \right\rangle_\Gamma + \left\langle\gamma_t^+{\bf curl\,E}_h,\jump{\pi_t\bs v} \right\rangle_\Gamma\Big)\,.
\end{align}
Now, picking test functions of the form $(\bs u,\bs 0)$ and $(\bs 0, \bs v)$ for $\bs u,\bs v$ infinitely differentialbe and compactly supported in $\mathbb R^3\setminus\Gamma$ in the expression above, we see that the following two equalities must hold independently:
\[
\left({\bf curl\,curl\,E}_h^{scat},\bs u\right)_{\mathbb R^3} +  (s/c_+)^2\left({\bf E}_h^{scat},\bs u\right)_{\mathbb R^3} = \; 0 \; =
\left({\bf curl\,curl\,E}_h,\bs v\right)_{\mathbb R^3} +  (s/c_-)^2\left({\bf E}_h,\bs v\right)_{\mathbb R^3}\,.
\]
Hence, equations \eqref{eq:EquivalentTPA} and \eqref{eq:EquivalentTPB} are verified, as the expressions above are simply their distributional forms. We can therefore drop the terms on the left hand side of \eqref{eq:ExpandedB} and rearrange the remaining terms to obtain
{\begin{equation}
\label{eq:zzz}
 \left\langle\mu_-^{-1}\gamma_t^+{\bf curl\,E}_h - \mu_+^{-1}\gamma_t^-{\bf curl\,E}_h^{scat},\jump{\pi_t\bs v} \right\rangle_\Gamma = -\left\langle\mu_+^{-1}\left(\jump{\gamma_t{\bf curl\,E}_h^{scat}}-\bs\lambda\right) +\mu_-^{-1}\jump{\gamma_t{\bf curl\, E}_h^{\phantom{t}}},\pi_t^+\bs u \right\rangle_\Gamma\,,
\end{equation} }
where we used the fact that for our test functions $\jump{\pi_t\bs u} = - \jump{\pi_t\bs v}$, and that, when testing with elements of ${\bf X}_h$, $\pi_t^-\bs u = \pi_t^+\bs v$. Since the tangential projection $\pi_t^+$ is surjective and we can take test functions of the form $(\bs u,\bs 0)$, from the expression above we conclude that for all $\bs\xi\in {\bf H}_{\perp}^{-1/2}(\mathrm{curl}_{\Gamma}, \Gamma)$
\[
\left\langle\mu_+^{-1}\left(\jump{\gamma_t{\bf curl\,E}_h^{scat}}-\bs\lambda\right) +\mu_-^{-1}\jump{\gamma_t{\bf curl\, E}_h^{\phantom{t}}},\bs \xi \right\rangle_\Gamma = 0\,,
\]
which implies \eqref{eq:EquivalentTPD}. This then forces the left hand side of \eqref{eq:zzz} to be equal to zero and thus, since $\jump{\pi_t\bs v}\in{\bf Y}_h$ is arbitrary, we obtain the second component of \eqref{eq:EquivalentTPF}. This concludes the proof. $\Box$
\end{proof}

\noindent\textbf{Well posedness.} We now prove the unique solvability of the variational formulation \eqref{eq:VariationalForm} and establish stability bounds for the solutions in terms of the Laplace parameter $s$. Due to the equivalence theorems proved in this section, this result will imply the solvability of the BIE system \eqref{eq:WeakBIE}. Moreover, the bounds obtained for the \textit{volume} unknowns ${\bf E}_h^{scat}$ and ${\bf E}_h$ from \eqref{eq:VariationalForm}  will result in similar bounds for the \textit{boundary} unknowns $\bs \phi_+$ and $\phi_-$ appearing in \eqref{eq:WeakBIE}. The process of sidestepping the direct analysis of the BIE in favor of a transmission problem and then post processing the result to recover estimates for the solutions of the BIE is similar in spirit to the technique introduced by N\'edelec and Planchard \cite{NePl1973}.

\begin{theorem}
\label{thm:1}
The boundary integral equation system \eqref{eq:WeakBIE} is uniquely solvable and the solution pair $({\bf j}_h,{\bf m}_h)$ satisfies the stability estimate
\begin{equation}
\label{eq:StabilityDensity}
 \|{\bf j}_h\|_{\|} + \|{\bf m}_h\|_{\perp} \lesssim \frac{|s|^2}{\sigma\underline{\sigma}^3}\left(\|\bs\lambda\|_{\|} +\|\bs\phi\|_{\perp}\right)\,.
\end{equation}
Moreover, $({\bf E}_h^{scat},{\bf E}_h)$ defined as in \eqref{eq:IntegralRepresentations} and solving \eqref{eq:VariationalForm} are bounded as 
\begin{equation}
\label{eq:Stability}
\triple{({\bf E}_h^{scat},{\bf E}_h)}_s \lesssim \frac{|s|^2}{\sigma\underline{\sigma}^2}\left(\|\bs\lambda\|_{\|} +\|\bs\phi\|_{\perp}\right)\,.
\end{equation}

\end{theorem}
\begin{proof}
We start by noting that
\begin{equation}
\label{eq:ellbound}
|\langle\mu_+^{-1}\bs\lambda,\pi_t^+\bs u\rangle_\Gamma| \lesssim \|\bs\lambda\|_{\|}\triple{\bs u}_1 \lesssim \frac{1}{\underline\sigma}\|\bs\lambda\|_{\|}\triple{\bs u}_s \lesssim \frac{1}{\underline\sigma}\|\bs\lambda\|_{\|}\triple{(\bs u,\bs v)}_s.
\end{equation}
Thus, in view of \eqref{eq:NormProperties}, we can apply the Lax-Milgram lemma to conclude that the problem is uniquely solvable in $\mathbb H_0^h$.  We must then build a particular pseudo inverse of the generalized trace $\pi$ that will allow us to take advantage of this.

Since the tangential projections $\pi_t^+$ and $\pi_t^-$ are surjective, given $(\bs \lambda,\bs\phi)\in {\bf H}_{\|}^{-1/2}({\rm div}_\Gamma,\Gamma) \times {\bf H}_\perp^{-1/2}({\rm curl}_\Gamma,\Gamma)$ we can find functions ${\bf E}_1$ and ${\bf E}_2$ such that
\[
\pi_t^+{\bf E}_1 = \bs\phi\,, \qquad \pi_t^-{\bf E}_1 = \bs 0\,, \qquad \text{ and } \qquad \pi_t^+{\bf E}_2 = \pi_t^-{\bf E}_2 = \bs 0\,. 
\]
So that $\pi^\dagger (\bs 0,\bs \phi) = ({\bf E}_1,{\bf E}_2)$. Note that this definition is independent of the discrete spaces. We then use this to construct functions
\[
{\bf E}_0^{scat} : = {\bf E}_h^{scat} - {\bf E}_1 \qquad \text{ and } \qquad {\bf E}_0 := {\bf E}_h-{\bf E}_2\,.
\]
By defining them in this way, we ensure that pair $({\bf E}_0^{scat},{\bf E}_0) = ({\bf E}_h^{scat},{\bf E}_h) - \pi^\dagger(\bs 0,\bs \phi)$ belongs to $\mathbb H_0^h$. 

Then, using \eqref{eq:NormPropertiesB} as starting point we compute
\begin{align*}
\sigma\triple{({\bf E}_0^{scat},{\bf E}_0)}_s^2 \lesssim\,& {\rm Re\,}\left[\,\overline s\, {\rm B}\left(({\bf E}_0^{scat},{\bf E}_0),(\overline{{\bf E}}_0^{scat},\overline{{\bf E}}_0)\right)\right] & \\
=\,& {\rm Re\,}\left[\,\overline s\, \left( - \langle\mu_+^{-1}\bs\lambda,\pi_t^+\overline{{\bf E}}_0^{scat}\rangle_\Gamma - {\rm B}\left(({\bf E}_1,{\bf E}_2),(\overline{{\bf E}}_0^{scat},\overline{{\bf E}}_0)\right)\right)\right] & \\
\lesssim\,& |s|\left(\frac{1}{\underline\sigma}\|\bs\lambda\|_{\|} + \triple{({\bf E}_1,{\bf E}_2)}_s\right)\triple{({\bf E}_0^{scat},{\bf E}_0)}_s \qquad \qquad & {\small \text{(using \eqref{eq:NormPropertiesA} and \eqref{eq:ellbound})}}\\
\lesssim\,& |s|\left(\frac{1}{\underline\sigma}\|\bs\lambda\|_{\|} + \frac{|s|}{\underline{\sigma}}\triple{({\bf E}_1,{\bf E}_2)}_1\right)\triple{({\bf E}_0^{scat},{\bf E}_0)}_s  & {\small \text{(using \eqref{eq:NormEquiv})}}\\
\lesssim\,& \frac{|s|^2}{\underline{\sigma}^2}\left(\|\bs\lambda\|_{\|} + \|\bs\phi\|_{\perp}\right)\triple{({\bf E}_0^{scat},{\bf E}_0)}_s\,. &
\end{align*}
Now, recalling that $({\bf E}_h^{scat},{\bf E}_h) = ({\bf E}_0^{scat},{\bf E}_0) + ({\bf E}_1,{\bf E}_2)$, the estimate above leads to
\[
\triple{({\bf E}_h^{scat},{\bf E}_h)}_s \lesssim \frac{|s|^2}{\sigma\underline{\sigma}^2}\left(\|\bs\lambda\|_{\|} + \|\bs\phi\|_{\perp}\right) + \triple{({\bf E}_1,{\bf E}_2)}_s \lesssim \frac{|s|^2}{\sigma\underline{\sigma}^2}\left(\|\bs\lambda\|_{\|} +\|\bs\phi\|_{\perp}\right) +  \frac{|s|}{\underline\sigma}\|\bs\phi\|_{\perp}\,
\]
which implies \eqref{eq:Stability}.

We now use the fact that the densities ${\bf j}_h$ and ${\bf m}_h$ are jumps in the traces of ${\bf E}_h$ and ${\bf curl\, E}_h$ to estimate their norms. For ${\bf j}_h$ we have
\begin{align*}
 \|\, {\bf j}_h \|_{\|} 
 & = \left\|\;  \jump{\, (s\mu_-)^{-1} \gamma_t {\bf curl\, E}_h}\; \right\|_{\|}  \qquad \quad & \text{{\small (From \eqref{eq:IntegralRepresentationsD})}} \\
&\lesssim  \frac{1}{|s|}\left(\| {\bf curl\,E}_h\|_{{\bf L}^2(\mathbb{R}^3\setminus\Gamma)} +  \| {\bf curl \, curl \,E}_h\|_{ {\bf  L}^2(\mathbb{R}^3 \setminus \Gamma)}\right) & \text{\small (Trace inequality)}
\\
& \lesssim\frac{1}{|s|}\left( \| {\bf curl\,E}_h\|_{{\bf L}^2(\mathbb{R}^3\setminus\Gamma)} + (|s|/c_-)^2 \| {\bf E}_h\|_{ {\bf L}^2(\mathbb{R}^3 \setminus \Gamma)}\right)   &  \text{\small (From \eqref{eq:EquivalentTPB})} \\
& \lesssim \frac{\max \{ 1, |s|\}}{|s|} \triple{ {\bf E}_h}_{s}  \\
&\lesssim  \frac{1}{\underline{\sigma}} \triple{ {\bf E}_h}_{s} \\
\intertext{ while for ${\bf m}_h$ it holds}
\| {\bf m}_h \|_{\perp} &  = \left\| \; \jump{ \pi_t{\bf  E}_h}\right\|_{\perp}  \lesssim \triple{{\bf E}_h}_{1} \lesssim  \frac{1}{\underline{\sigma}} \, \triple{ {\bf E}_h}_{s}\,.
\end{align*}
Therefore, the bound \eqref{eq:StabilityDensity} is obtained by combining the previous two inequalities with  \eqref{eq:Stability}. $\Box$
\end{proof}
Let $\mathbb X$ be a Banach space; we say that an analytic function $A : \{s\in\mathbb C: {\rm Re\,}s>0\} \to \mathbb X$ is a \textit{hyperbolic symbol} of order $\mu\in\mathbb R$ and write $A\in \mathcal A(\mu,\mathbb X)$ if:
\begin{enumerate}
\item There exists a  positive, non increasing function $C_A:(0,\infty)\to(0,\infty)$ such that for all $\sigma\in(0,1]$
\[
C_A(\sigma) \leq K\sigma^{-\ell} \qquad \qquad \text{ for } K>0 \;\text{ and } \ell\geq 0\,.
\]

\item $A(s)$ is bounded as
\[
\|A(s)\| \leq C_A\left({\rm Re\,}s\right) |s|^\mu\,.
\]
\end{enumerate}
We can formulate the results that we have proven so far in terms of hyperbolic symbols as follows.
\begin{corollary}
\label{cor:symbols}
If we denote by
\[
{\rm S}(s):  {\bf H}_{\|}^{-1/2}({\rm div}_\Gamma,\Gamma) \times {\bf H}_\perp^{-1/2}({\rm curl}_\Gamma,\Gamma) \to {\bf X}_h\times{\bf Y}_h
\]
the solution operator taking problem data $(\bs\lambda,\bs\phi)$ and mapping it to the densities $({\bf j}_h,{\bf m}_h)$ satisfying \eqref{eq:WeakBIE}, and by 
\[
{\rm P(s)}: {\bf H}_{\|}^{-1/2}({\rm div}_\Gamma,\Gamma) \times {\bf H}_\perp^{-1/2}({\rm curl}_\Gamma,\Gamma) \to \mathbb H^h
\]
the operator mapping the data  $(\bs\lambda,\bs\phi)$ to the functions $({\bf E}_h^{scat},{\bf E}_h)$ satisfying the distributional system \eqref{eq:LDsystem} and obtained from the densities $({\bf j}_h,{\bf m}_h)$ through the post-processing step \eqref{eq:IntegralRepresentations}, then
\begin{subequations}
\begin{alignat}{6}
\label{eq:SymbolS}
{\rm S}(s) \in\,& \mathcal A(2, {\bf X}_h\times{\bf Y}_h) &\qquad \text{ and } \qquad C_{\rm S} = \frac{1}{\sigma\underline\sigma^3}\,, \\[.5ex]
\label{eq:SymbolP}
{\rm P}(s) \in\,& \mathcal A(2, \mathbb H^h) &\qquad \text{ and } \qquad C_{\rm P} = \frac{1}{\sigma\underline\sigma^3}\,.
\end{alignat}
\end{subequations}
\end{corollary}
We point out that, in contrast with the bound obtained in Theorem \ref{thm:1},  the function $C_{\rm P}$ above has an additional factor of $\underline\sigma^{-1}$. This is due to the fact that the inequality \eqref{eq:Stability} in Theorem \ref{thm:1} involves the energy norm on the left hand side, while the definition of $C_{\rm P}$ is done with respect to the natural norm in $\left({\bf H}({\bf curl},\mathbb R^3\setminus\Gamma)\right)^2$. Thus, the additional factor is introduced via $\eqref{eq:NormEquiv}$.
%
\section{Time domain results}\label{sec:4}
%
We will now return to the time domain. Notationwise, all the functions appearing in this section are the time domain counterparts of those appearing in the previous two sections, $\mathcal L$ denotes the Laplace transform operator, $*$ is the convolution operator, and $\mathbb X, \mathbb Y$ are Banach spaces.

The estimates we have obtained so far can be transformed into time-domain statements thanks to the following result due to Sayas \cite[Proposition 3.2.2]{Sayas2016errata,Sayas2016} that we present here without proof. In the following statement,
\begin{theorem}[\cite{Sayas2016errata}] \label{thm:2}
Let $A = \mathcal{L}\{a\} \in \mathcal{A} (k + \alpha, \mathcal{B}(\mathbb X,\mathbb Y))$ with $\alpha\in [0, 1)$ and $k$ a non-negative integer.  If $ g \in \mathcal{C}^{k+1}(\mathbb{R}, \mathbb X)$ is causal and its derivative $g^{(k+2)}$ is integrable, then $a* g \in \mathcal{C}(\mathbb{R}, \mathbb Y)$ is causal and 
\[
\| (a*g)(t) \|_{\mathbb Y} \le 2^{\alpha} C_{\epsilon} (t) C_A (t^{-1}) \int_0^1 \|(\mathcal{P}_2g^{(k)})(\tau) \|_{\mathbb X} \; d\tau,
\]
where 
\[
C_{\epsilon} (t) := \frac{1}{2\sqrt{\pi}} \frac{\Gamma(\epsilon/2)}{\Gamma\left( (\epsilon+1)/2 \right) } \frac{t^{\epsilon}}{(1+ t)^{\epsilon}}, \qquad (\epsilon :=  1- \alpha \; \;  \mbox{and}\; \;  \mu = k +\alpha)
\]
and 
\[
(\mathcal{P}_2g) (t) :=  g + 2\dot{g} + \ddot{g}.
\]
\end{theorem}
Given \textbf{time domain} data $(\gamma_t^+{\bf curl\,E}^{inc},\pi_t^+{\bf E}^{inc})$ the solution to the time domain system \eqref{eq:TDsystem} and to the time-domain counterpart of \eqref{eq:WeakBIE} are given respectively by
\begin{align*}
({\bf E}_h^{inc},{\bf E}_h)(t) =\,& \mathcal L^{-1}\left\{{\rm P}(s)\right\}*(\gamma_t^+{\bf curl\,E}^{inc},\pi_t^+{\bf E}^{inc})\,, \\[.5ex]
({\bf j}_h,{\bf m}_h)(t) =\,& \mathcal L^{-1}\left\{{\rm S}(s)\right\}*(\gamma_t^+{\bf curl\,E}^{inc},\pi_t^+{\bf E}^{inc}).
\end{align*}
Hence, in view of the theorem above and of Corollary \ref{cor:symbols}, we have the following stability result in the time-domain:
\begin{theorem}[Time-domain stability]
\label{thm:TimeDomain}
If for every $t\in(0,T)$ the time domain data $(\gamma_t^+{\bf curl\,E}^{inc},\pi_t^+{\bf E}^{inc})$ considered as a mapping from the interval $[0,T]$ to the Banach space
\[
\mathbb X := {\bf H}_{\|}^{-1/2}({\rm div}_\Gamma,\Gamma) \times {\bf H}_\perp^{-1/2}({\rm curl}_\Gamma,\Gamma),
\]
belongs to $\mathcal C^{3}\left([0,T], \mathbb X\right)$ and its fourth derivative is integrable, then the unique time-domain pair $({\bf E}_h^{inc},{\bf E}_h)$ satisfying \eqref{eq:TDsystem} is causal and continuous as a mapping from $[0,T]$ to the Banach space
\[
\mathbb H^h :=\left\{(\bs u,\bs v)\in\left({\bf H}({\bf curl}^2,\mathbb R^3\setminus\Gamma)\right)^2 :  \jump{\gamma_t\bs v}\in {\bf X}_h \;\; \text{ and } \;\;\jump{\pi_t\bs u}\in {\bf Y}_h \right\}\,,
\]
and satisfies the bound
\[
\|({\bf E}_h^{inc},{\bf E}_h)(t)\|_{\mathbb H}\lesssim \frac{t^2\max\{1,t^3\}}{1+t}\int_0^t\|\mathcal P_2\left((\gamma_t^+{\bf curl\,E}^{inc},\pi_t^+{\bf E}^{inc})\right)^{(2)}\|_{\mathbb X}\,d\tau\,.
\]
Moreover, the densities $({\bf j}_h,{\bf m}_h)$ satisfying the boundary integral system \eqref{eq:WeakBIE} are continuous and causal and bounded by
\[
\|({\bf j}_h^,{\bf m}_h)(t)\|_{{\bf X}_h\times{\bf Y}_h}\lesssim \frac{t^2\max\{1,t^3\}}{1+t}\int_0^t\|\mathcal P_2\left((\gamma_t^+{\bf curl\,E}^{inc},\pi_t^+{\bf E}^{inc})\right)^{(2)}\|_{\mathbb X}\,d\tau\,.
\]
\end{theorem}

\noindent\textbf{Continuous vs. discrete.} We conclude this section by remarking that estimates for the continuous problem can be obtained as corollaries of Propositions \ref{pr:1} and \ref{pr:2} as well as from Theorem \ref{thm:1} by simply letting ${\bf X}_h = {\bf H}_{\|}^{-1/2}({\rm div}_\Gamma,\Gamma) $ and ${\bf Y}_h = {\bf H}_\perp^{-1/2}({\rm curl}_\Gamma,\Gamma)$. However, if ${\bf X}_h$ and ${\bf Y}_h$ are strict closed subspaces, error estimates for the corresponding conforming Galerkin estimates can be easily obtained from the previous analysis, as we will demonstrate in the final section.

\section{Discrete and semi discrete error estimates.}\label{sec:5}

Due to the linearity of the model, the framework developed in the previous sections can be easily transferred into error estimates for Galerkin semi discretizations in space, and full discretizations using Convolution Quadrature for time discretization. Since most of the arguments needed for the error analysis are straightforward extensions of what has been done in the previous sections, we will quickly go through the steps off the process without stopping on the full details.

We start by defining the error functions
\[
{\bf e}_h : = {\bf E}-{\bf E}_h \qquad \text{ and } \qquad {\bf e}_h^{scat} : = {\bf E}^{scat}-{\bf E}_h^{scat} \,, 
\]
and the error densities
\[
{\bf e_{j}} : = {\bf j} - {\bf j}_h \qquad \text{ and } \qquad {\bf e_{m}} : = {\bf m} - {\bf m}_h\,.
\]
From these definitions it follows that
\[
\begin{array}{lll}
\jump{\gamma_t{\bf curl\, e}_h^{\phantom{s}}} = s\mu_-{\bf e_j} \,, & \phantom{++++} & \jump{\pi_t{\bf e}_h^{\phantom{s}}} = {\bf e_m}\,, \\[1.2ex]
\jump{\gamma_t{\bf curl\, e}_h^{scat}} : = (\mu_+/\mu_-)\jump{\gamma_t{\bf curl\, e}_h^{\phantom{s}}}\,, & \phantom{++++} & \jump{\pi_t{\bf e}_h^{scat}} = -\jump{\pi_t{\bf e}_h\phantom{s}}\,,
\end{array}
\]
so that, just as before, we can eliminate the variables related to the scattered wave from the system. \\

\noindent\textbf{The Laplace-domain error equations.} Observing that the error functions can be recovered from the error densities through the integral representations
\[
{\bf e}_h =  (s\epsilon_-)^{-1}\mathcal S(s/c_-){\bf e_j} - \widetilde{\mathcal D}(s/c_-){\bf e_m}\,,\\ \qquad {\bf e}_h^{scat} = \widetilde{\mathcal D}(s/c_+){\bf e_m} -(s\epsilon_+)^{-1}\mathcal S(s/c_+){\bf e_j}\,,
\]
it becomes clear that the error densities ${\bf e_{j}} $ and ${\bf e_{m}} $ satisfy the BIE
\[
\langle \mathbb L(s)({\bf e_{j}},{\bf e_m})^\top, (\bs\eta,\bs\varphi)\rangle_\Gamma = 0\;\;\;\; \forall (\bs\eta,\bs\varphi)\in  {\bf X}_h\times{\bf Y}_h\,,
\]
where $\mathbb L(s)$ is defined s in \eqref{eq:MatricesA}. This equation is in turn equivalent to the transmission problem
\begin{align*}
{\bf curl\,curl\, e}_h^{scat} + (s/c_+)^2{\bf e}_h^{scat} =\,& \bs 0  \qquad \qquad \qquad \text{ in } \mathbb R^3\setminus\Gamma\,, \\
{\bf curl\,curl\, e}_h + (s/c_-)^2{\bf e}_h =\,& \bs 0  \qquad \qquad  \qquad \text{ in } \mathbb R^3\setminus\Gamma\,, \\
\jump{\pi_t{\bf e}_h^{scat}} + \jump{\pi_t{\bf e}_h^{\phantom{t}}} =\,& \bs0 \,,\\
\mu_+^{-1}\jump{\gamma_t{\bf curl\,e}_h^{scat}} + \mu_-^{-1}\jump{\gamma_t{\bf curl\,e}_h^{\phantom{t}}} =\,& \bs 0 \,,\\
\left(\jump{\gamma_t{\bf curl\,e}_h^{\phantom{t}}} - s\mu_-{\bf j}\,, \jump{\pi_t{\bf e}_h^{\phantom{t}}}-{\bf m}\right) \in\,& {\bf X}_h\times{\bf Y}_h\,, \\
\left(\pi_t^+{\bf e}_h - \pi_t^-{\bf e}_h^{scat}\,,\,\, \mu_-^{-1}\gamma_t^+{\bf curl\,e}_h - \mu_+^{-1}\gamma_t^-{\bf curl\,e}_h^{scat} \right) \in\,& {\bf X}_h^\circ \times {\bf Y}_h^\circ\,,
\end{align*} 
which can be posed variationally as:
\begin{align}
\nonumber
& \text{Given } ({\bf j},{\bf m})\in {\bf H}_{\|}^{-1/2}({\rm div}_\Gamma,\Gamma) \times {\bf H}_\perp^{-1/2}({\rm curl}_\Gamma,\Gamma)\,,\, \\
\nonumber
& \text{find } ({\bf e}_h^{scat},{\bf e}_h)\in \left({\bf H}({\bf curl}^2,\mathbb R^3\setminus\Gamma)\right)^2 \text{ such that: } \\[-.4ex]
\nonumber
&\left(\jump{\gamma_t{\bf curl\,e}_h^{\phantom{t}}} - s\mu_-{\bf j}, \jump{\pi_t{\bf e}_h^{\phantom{t}}}-{\bf m}\right) \in {\bf X}_h\times{\bf Y}_h \;\text{ and }\\[-.4ex]
\label{eq:ErrorVariational}
& {\rm B}\left(({\bf e}_h^{scat},{\bf e}_h),(\bs u,\bs v)\right) = 0 \quad \forall\,(\bs u,\bs v)\in \mathbb H^h\,.
\end{align}
%
\noindent\textbf{Error estimates.} Using the variational form, we can prove the following:
\begin{proposition}[Laplace domain error estimates]
\label{prop:LDerror}
The variational problem for the error functions is uniquely solvable and we have the following error estimates:
\begin{subequations}
\label{eq:SemiErrorEstimates}
\begin{align}
\label{eq:SemiErrorEstimatesA}
\triple{({\bf e}_h^{scat},{\bf e}_h)}_s \lesssim\,& \frac{|s|^2}{\sigma\underline\sigma}\left(\|s\,({\bf j}-{\bf j}^*_h)\|_{\|} + \|{\bf m}-{\bf m}_h^*\|_{\perp}\right)\,, \\
\label{eq:SemiErrorEstimatesB}
\triple{({\bf e}_h^{scat},{\bf e}_h)}_1 \lesssim\,& \frac{|s|^2}{\sigma\underline\sigma^2}\left(\|s\,({\bf j}-{\bf j}^*_h)\|_{\|} + \|{\bf m}-{\bf m}_h^*\|_{\perp}\right)\,,\\
\label{eq:SemiErrorEstimatesC}
\triple{({\bf e_j},{\bf e_m})}_{ \|\times\perp} \lesssim\,& \frac{|s|^2}{\sigma\underline\sigma^2}\left(\|s\,({\bf j}-{\bf j}^*_h)\|_{\|} + \|{\bf m}-{\bf m}_h^*\|_{\perp}\right)\,,
\end{align}
\end{subequations}
where $({\bf j}^*_h,{\bf m}_h^*)$ are the best approximations to $({\bf j},{\bf m})$ in the discrete space ${\bf X}_h\times{\bf Y}_h$.
\end{proposition}
\begin{proof}
Unique solvability follows from the ellipticity estimate \eqref{eq:NormPropertiesB} and Lax-Milgram's lemma, while from \eqref{eq:NormPropertiesA} and \eqref{eq:NormEquiv} we obtain the bound
\begin{equation}
\label{eq:Continuity2norms}
\left|{\rm B}\left((\bs u,\bs v), (\widetilde{\bs u},\widetilde{\bs v})\right)\right| \lesssim \triple{(\bs u, \bs v)}_s \triple{(\widetilde{\bs u}, \widetilde{\bs v})}_s \leq \frac{|s|}{\underline\sigma} \triple{(\bs u, \bs v)}_{1} \triple{(\widetilde{\bs u}, \widetilde{\bs v})}_s\,.
\end{equation}
Now, from the continuity of the pseudo inverses of the tangential trace $\gamma_t$ and the tangential projection $\pi_t$, we can find functions $(\bs u,\bs v) \in \left({\bf H}({\bf curl}^2,\mathbb R^3\setminus\Gamma)\right)^2$ such that $\jump{\gamma_t\bs u} = - s\mu_-{\bf j}\,$,  $\jump{\pi_t\bs v}= -{\bf m}\,$, and
\begin{equation}
\label{eq:InverseContinuity}
 \quad \triple{(\bs u,\bs v)}_{1} \lesssim \left(\|s\,{\bf j}\|_{\|} +\|{\bf m})\|_{\perp}\right)\,.
\end{equation}
Hence, $({\bf e}_h^{scat}+\bs u,{\bf e}_h+\bs v)\in \mathbb H^h$ and we have
\begin{align*}
\triple{({\bf e}_h^{scat}+\bs u,{\bf e}_h+\bs v)}_s^{2} \lesssim\,& \frac{|s|}{\sigma}\left|{\rm B\,}\left(({\bf e}_h^{scat}+\bs u,{\bf e}_h+\bs v),\overline{({\bf e}_h^{scat}+\bs u,{\bf e}_h+\bs v)}\right)\right| & {\small\text{(By \eqref{eq:NormPropertiesB})}}\,, \\
=\,&  \frac{|s|}{\sigma}\left|{\rm B\,}\left((\bs u,\bs v),\overline{({\bf e}_h^{scat}+\bs u,{\bf e}_h+\bs v)}\right)\right| & {\small\text{(By \eqref{eq:ErrorVariational})}}\,, \\
\lesssim\,&  \frac{|s|^2}{\sigma\underline\sigma} \triple{(\bs u, \bs v)}_1\triple{({\bf e}_h^{scat}+\bs u,{\bf e}_h+\bs v)}_{s}  &  {\small\text{(By \eqref{eq:Continuity2norms})}}\,.
\end{align*}
From here, Cea's lemma and \eqref{eq:InverseContinuity} imply 
\[
\triple{({\bf e}_h^{scat},{\bf e}_h)}_s \lesssim \frac{|s|^2}{\sigma\underline\sigma}\left(\|s\,{\bf j}\|_{\|} + \|{\bf m}\|_{\perp}\right)\,,
\]
while \eqref{eq:SemiErrorEstimatesA} follows from the inequality above by letting the data in the variational problem for the error functions be the discrete approximation error $({\bf j}-{\bf j}^*_h,{\bf m}-{\bf m}_h^*)$. Finally, \eqref{eq:SemiErrorEstimatesA} together with \eqref{eq:NormEquiv} imply \eqref{eq:SemiErrorEstimatesB}, which in turn implies \eqref{eq:SemiErrorEstimatesC} in view of the continuity of the tangential projection and the tangential trace and the fact that
\[
s\mu_-{\bf e_j}=\jump{\gamma_t{\bf curl\,e}_h} \qquad \text{ and } \qquad {\bf e_m}=\jump{\pi_t{\bf e}_h}. \qquad  \qquad \Box
\]
\end{proof}
Note that, while we could have bounded
\[
\|(s\,({\bf j}-{\bf j}^*_h),{\bf m}-{\bf m}_h^*)\|_{\|\times\perp} \leq \frac{|s|}{\underline\sigma}\|({\bf j}-{\bf j}^*_h,{\bf m}-{\bf m}_h^*)\|_{\|\times\perp}\,,
\]
in the results above, keeping the product by $s$ localized to the term $({\bf j}-{\bf j}^*_h)$ will allow us to relax the time regularity requirements for the solution once we convert this result into the time domain, as we shall do now.

Using the bounds in Proposition \ref{prop:LDerror}, in particular the estimates \eqref{eq:SemiErrorEstimatesB} and \eqref{eq:SemiErrorEstimatesC}, in a manner analogous to the passage from Corollary \ref{cor:symbols} to Theorem \ref{thm:TimeDomain} we can prove the following time stability result for the Galerkin semi discretization. As usual, when in the time domain, we use the same notation to denote the time-domain functions whose Laplace transform is studied.
\begin{theorem}[Semi-discrete time stability]
If for every time $t\in[0,T]$ the solutions to \eqref{eq:WeakBIE} are such that
\begin{alignat*}{6}
{\bf j}(t) \in\,& \mathcal C^4\left([0,T],{\bf H}_{\|}^{-1/2}({\rm div}_\Gamma,\Gamma)\right)\,, &\qquad \text{ and } \qquad && {\bf j}^{(5)}(t)\in {\bf L}^1\left([0,T],{\bf H}_{\|}^{-1/2}({\rm div}_\Gamma,\Gamma)\right)\,, \\
{\bf m}(t)   \in\,& \mathcal C^3\left([0,T],{\bf H}_{\perp}^{-1/2}({\rm curl}_\Gamma,\Gamma)\right)\,, &\qquad \text{ and } \qquad && {\bf m}^{(4)}(t)\in {\bf L}^1\left(([0,T],{\bf H}_{\perp}^{-1/2}({\rm curl}_\Gamma,\Gamma)\right)\,,
\end{alignat*}
then the error functions $({\bf e}_h^{scat},{\bf e}_h)$ and the error densities $({\bf e_j},{\bf e_m})$ are causal, continuous, and for every $t\geq 0$ we have the estimates
\begin{align*}
\triple{({\bf e}_h^{scat},{\bf e}_h)(t)}_1 \lesssim\,& \frac{t\max\{1,t^2\}}{1+t}\int_0^t\|\mathcal P_2\left(\partial_t({\bf j}-{\bf j}_h^*),{\bf m}-{\bf m}_h^*\right)^{(2)}(t)\|_{\|\times\perp}\,d\tau\,,\\
\|({\bf e_j},{\bf e_m})(t)\|_{\|\times\perp} \lesssim\,& \frac{t\max\{1,t^2\}}{1+t}\int_0^t\|\mathcal P_2\left(\partial_t({\bf j}-{\bf j}_h^*),{\bf m}-{\bf m}_h^*\right)^{(2)}(t)\|_{\|\times\perp}\,d\tau\,,
\end{align*}
where $({\bf j}^*_h,{\bf m}_h^*)$ are the best approximations to $({\bf j},{\bf m})$ in the discrete space ${\bf X}_h\times{\bf Y}_h$.
\end{theorem}
Finally, Proposition 4.6.1 from \cite{Sayas2016} (which we will not present here) allows us to combine all the previous results into a time-domain error estimate for a fully discrete formulation. If a Convolution Quadrature scheme  of order $p$ with time step of size $\kappa$ was used to deal with time evolution, we will denote the fully discrete error functions by $({\bf e}_{h,\kappa}^{scat},{\bf e}_{h,\kappa})$, and we have the following estimate.
\begin{theorem}[CQ Fully discrete error estimate] 
If Convolution Quadrature of order $p$ with time step of size $\kappa$ is used for time discretization and $(\gamma_t^+{\bf curl\, E}^{inc},\pi_t^+{\bf E}^{inc})(t)$ are causal problem data satisfying 
\begin{align*}
(\gamma_t^+{\bf curl\, E}^{inc},\pi_t^+{\bf E}^{inc})(t) \in\,&  \mathcal C^{\,p+3}\left([0,T],\left({\bf H}({\bf curl}^2,\mathbb R^3\setminus\Gamma)\right)^2\right)\,,\\
\partial_t^{\,p+4}(\gamma_t^+{\bf curl\, E}^{inc},\pi_t^+{\bf E}^{inc})(t) \in\,& {\bf L}^1\left([0,T],\left({\bf H}({\bf curl}^2,\mathbb R^3\setminus\Gamma)\right)^2\right)\,.
\end{align*}
\end{theorem}
then
\[
\triple{({\bf e}_{h,\kappa}^{scat},{\bf e}_{h,\kappa})(t)}_1 \lesssim \kappa^p(1+t^3)\int_0^t\|\left(\partial_t^{\,p+4}({\bf j}-{\bf j}_h^*),\partial_t^{\,p+3}\left({\bf m}-{\bf m}_h^*\right)\right)(t)\|_{\|\times\perp}\,d\tau\,.
\]
\noindent \textbf{Convolution quadrature in a nutshell.} Convolution Quadrature (CQ), first proposed by Christian Lubich \cite{Lubich1988,Lubich:1988b}, constitutes a very powerful tool for discretization of integral equations in the time domain using a combination of time-domain problem data, and Laplace-domain integral kernels. For the sake of completenness, we offer here a very short and informal discussion of the ideas behind CQ. 

Time-domain integral equations have a convolutional structure, and thus, the solution process can be thought of as computing the convolution of the inverse of an integral operator with the problem data. The strength of CQ relies on the observation that the convolution of a causal function $g$ (the time-domain problem data) with a causal Laplace-transformable kernel $f^{-1}$ (the inverse of an integral operator) can be written, formally for some $\sigma>0$, in terms of the Laplace transform $F^{-1}:=\mathcal L\{f^{-1}\}$ as
\begin{equation}\label{eq:CQ1}
\left(f^{-1}*g\right)(t)=\frac{1}{2\pi i }\int_{\sigma-i\infty}^{\sigma+i\infty}\int_0^t F(s)^{-1} e^{s(t-\tau)}g(\tau)\,d\tau\, ds = \frac{1}{2\pi i }\int_{\sigma-i\infty}^{\sigma+i\infty}F(s)^{-1}y(t;s)\, ds,
\end{equation}
where $y(t;s)$ is the unique solution of
\[
\dot y(t)-sy(t)=g(t), \quad y(0)=0.
\]
The complex contour integral on the right hand side of \eqref{eq:CQ1} can then be approximated by quadrature to obtain
\[
\left(f^{-1}*g\right)(t) \approx  \frac{1}{2\pi i }\sum_{i=1}^n \omega_i F(s_i)^{-1}y(t;s_i)\,
\]
where the pairs $(\omega_i,s_i)$ are weights and points in a quadrature rule. For every point $s_i$ in the quadrature formula, the integrand is evaluated by producing a numerical solution $y_h(t;s_i)$ to the equation
\[
\dot y(t)-s_iy(t)=g(t), \quad y(0)=0.
\]
The solution $y_h(t;s_i)$ is then used as the right hand side of the Laplace-domain boundary integral equation
\[
F(s_i)\psi(t) = y_h(t;s_i),
\]
which is then solved numerically to produce the required sample of the integrand $\psi(t):=F(s_i)^{-1}y_h(t;s_i)$. The differential equation can be solved approximately by any ODE-solving procedure using time-domain information from $g$ and the Laplace-domain kernels for boundary integral equations are typically more tractable numerically than their time-domain counterparts. This procedure has also the advantage of employing well established computational technology for the solution of ordinary differential equations and frequency domain boundary integral equations. Beyond Lubich's original articles, many excellent resources are available for the reader interested in the details of the technique; in particular \cite{Hassell2016} offers a very accessible treatment of both the theoretical and implementation aspects, while \cite{Sayas2016} offers a comprehensive treatment. 
\section*{Acknowledgments}
%
Tonatiuh S\'anchez-Vizuet has been partially funded by the U. S. National Science Foundation through the grant NSF-DMS-2137305.

%
\bibliographystyle{abbrv}
\bibliography{references}

\end{document}